\documentclass[]{article}

\usepackage{amssymb, amsmath, amsthm}
\usepackage{geometry}
\usepackage{mathrsfs}

\theoremstyle{plain}
\newtheorem{theorem}{Theorem}[section] 
\newtheorem{lemma}[theorem]{Lemma}
\newtheorem{proposition}[theorem]{Proposition}
\newtheorem{corollary}[theorem]{Corollary}

\theoremstyle{definition}
\newtheorem{definition}[theorem]{Definition}

\newtheorem{ass}[theorem]{Assumption}

\def\d{\mathrm{d}} 
 
\def\eps{{\varepsilon}} 
\def\ue{u^\eps}  
\def\De{D^\eps}  
\def\R{{\mathbb{R}}}
\def\C{\mathbb{C}}
\def\Z{{\mathbb{Z}}}
\def\N{{\mathbb{N}}}
\def\Te{\mathcal{T}^\eps}
\def\p{{\partial}}
\def\Id{{\mathrm{Id}}}
\def\O{{\Omega}}
\def\grad{\nabla}
\def\div{\mathrm{div}}
\newcommand{\weak}{\longrightharpoonup\,}
\newcommand{\longrightharpoonup}{\mathrel{\relbar\joinrel\rightharpoonup}}
\newcommand{\nx}[2]{\left\lVert#1\right\rVert_\text{{$#2$}}}
\newcommand{\ud}{\ \mathrm{d}}
\newcommand{\udvol}{\ \mathrm{d}\mathrm{vol}_M}
\renewcommand{\d}{\mathrm{d}}

\begin{document}




\title{An Application of Periodic Unfolding on Manifolds}


\author{S\"oren Dobbersch\"utz\thanks{Nano-Science Center, University of Copenhagen, Universitetsparken 5,
2100 K\o{}benhavn, Denmark. \mbox{E-Mail:}~\texttt{sdobber@nano.ku.dk}}}

\maketitle

\begin{abstract}
We show how the newly developed method of Periodic Unfolding on Riemannian manifolds can be applied to PDE problems: We consider the homogenization of an elliptic model problem. In the limit, we obtain a generalization of the well-known limit- and cell-problem. By constructing an equivalence relation of atlases, one can show the invariance of the limit problem with respect to this equivalence relation. This implies e.g.\ that the homogenization limit is independent of change of coordinates or scalings of the reference cell. \\


\end{abstract}





\section{Introduction}

This paper deals with applications of Periodic Unfolding on Riemannian manifolds. The method has been developed in \cite{dobo_perunf} and \cite{dobo_constrpumf} to allow for periodic homogenization of problems posed on "nonflat" objects like surfaces or spherical zones (as opposed to "flat" domains in $\R^n$). Prospective applications come from the field of catalysis: For example in the design and operation of heterogeneous catalysts for the depollution of automotive exhaust gases, the surface structure and -composition plays an important role for the effectiveness of the catalytic process, see e.g.\ \cite{sksu_spraypyrol} for such findings or \cite{hetcat} for an overview of the field. However, since the chemical activity takes place at the nonflat boundary of these microscopic porous particles, one needs methods to derive effective descriptions of these materials, for example for the implementation of efficient numerical simulations.

The field of (mathematical) homogenization emerged around the 1970's, see Bensoussan, Lions, and Papanicolao~\cite{blp_asymtanal} as well as Sanchez-Palencia~\cite{sanpal}. 
In general, one considers a family of operators $\mathcal{L}^\eps$ for a scale parameter $\eps>0$, and tries to find a limit $u_0$ of the sequence $u^\eps$ and an operator $\mathcal{L}_0$ for $\eps\rightarrow 0$, such that the solutions of the operator equation $\mathcal{L}^\eps u^\eps = 0$ converge to $u_0$ and fulfill $\mathcal{L}_0 u_0 = 0$ (in some specified sense). This problem is then considered to be an ''effective'' or approximate description of the original problem for small $\eps$.

Several methods have been devised to study such processes effectively in the case that $\mathcal{L}^\eps$ is a family of partial differential operators (without claiming to be exhaustive): Starting from the heuristic method of asymptotic expansion, Tartar developed the method of oscillating test functions to allow a mathematical proof of convergence (see e.g.\ \cite{blp_asymtanal}). Being a rather tedious method, the proof of convergence was then facilitated by the invention of the notion of two-scale convergence by Nguetseng \cite{ng_2s1} and Allaire \cite{al_2s2}. See also \cite{cido_intrhom} for an introduction to the field. 

A more recent development is the technique of Periodic Unfolding, introduced by Cioranescu, Damlamian and Griso \cite{cidagr_perunhom} based on an idea of Arbogast, Douglas and Hornung \cite{ardoho_dbl-poros}. With the help of this method, one can replace the two-scale convergence by well-established notions of convergence (like weak and strong convergence) in Banach spaces. The reader is referred to \cite{cidagr_pu} for a well-written introduction to Periodic Unfolding and its various applications.

However, one drawback of all methods derived so far is that they can only be applied to (flat) subsets of the Euclidean space $\R^n$. Although there are some generalizations to surfaces available, see \cite{aldaho_2s} and \cite{cidoza_perunperf}, they can not be applied to surfaces having itself a periodic structure. Therefore we extended the notion of Periodic Unfolding to certain compact Riemannian manifolds, see \cite{dobo_constrpumf}, \cite{dobo_perunf}, and Section \ref{sec:pumf}. In this work, we show how the theoretical results can be applied to the standard elliptic homogenization problem. Although the limit problem (which is derived in Section~\ref{sec:applell}) looks very similar to the usual limit obtained in the case of a domain in $\R^n$ (compare e.g.\ \cite{cido_intrhom}), we would like to point the reader to the following particularities: Our result is not restricted to the case of a ''flat'' domain, but also applies for example to suitable curved surfaces. We also admit general Riemannian metrics and thus extend the known results, where implicitly only the standard (Euclidean) metric has been used. (This can also be seen as a stationary version of homogenization with evolving microstructure as developed by Peter \cite{pe_domevolve}.) Moreover, we show that the limit is independent with respect to the choice of equivalent atlases. This and the coresponding equivalence relation is dealt with in 
Section~\ref{sec:RMequivatl}.

\section{Periodic Unfolding on Manifolds}
\label{sec:pumf}

\subsection{Notation}

Let $M\subset \R^m$ be a smooth compact oriented Riemannian manifold (in the sense of Schwarz~\cite{schw_hodge}) with $\dim(M)=n$; $n,m\in \N$. We allow $M$ to be with or without boundary. The tangent space $T_xM$ at $x\in M$ gives rise to the tangent bundle $TM$. Let there be given a smooth Riemannian metric $g_M\in \Gamma(T^*M \otimes T^*M)$, i.e. a smooth section of the product of the cotangent bundle $T^*M$. For $x\in M$, $g_M(x)$ defines a scalar product on the tangent space $T_xM$. Moreover, we consider a finite atlas $\mathscr{A}=\{ (U_\alpha,\phi_\alpha);\alpha \in I \}$ with index set $I$. (Note that we do not consider equivalence classes of atlases here, but we always refer to a fixed one.) The corresponding partition of unity subordinate to $\{ U_\alpha;\alpha\in I \}$ is denoted by $\{ \pi_\alpha;\alpha \in I \}$. A chart $\phi_\alpha=(x^1,\dots, x^n):U_\alpha \rightarrow \R^n$ induces a set of tangent vectors, the local basis vectors $\frac{\p}{\p x^i}$, $i=1,\dots,n$. They allow a representation of $g_M$ in local coordinates as $g_M=\sum_{i,j=1}^n g_{ij} \,\mathrm{d}x^i \otimes \mathrm{d}x^j$, where $g_{ij}= g_M(\frac{\p}{\p x^i}, \frac{\p}{\p x^j})$. By defining the matrix $G=[g_{ij}]_{i,j=1,\dots,n}$, the volume form on $M$ is given by $\mathrm{dvol}_M:=\sqrt{|\det G|} \ud x^1\wedge\dots\wedge\ud x^n$. Similarly, $\phi_\alpha$ also induces pushforward and pullback operators (designated as $(\phi_\alpha)_*$ and $(\phi_\alpha)^*= (\phi_\alpha^{-1})_*$, resp.) for functions, forms and vector fields. E.g.\ for a function $f:U_\alpha \rightarrow \R$, one has $\phi_\alpha f = f\circ \phi_\alpha^{-1}$. 
Details concerning these notions and Riemannian manifolds can be found in \cite{globan}, for example.


Lebesgue-spaces on $M$ are as usual denoted by $L^p(M)$, where $p\in [1,\infty]$ denotes the order of integrability. Similarly, we denote the Lebesgue space of sections of the tangent bundle by $L^p TM$. The Hilbert space of one time weakly differentiable functions in $L^2(M)$ will be denoted by $H^1(M)$.  The subscript $\#$ indicates periodicity, and the subscript $0$ designates a vanishing trace (in the case of $H^1_0(M)$) or a compact support in the case of spaces of differentiable functions.

Denote by $Y:=[0,1)^n$ the reference cell in $\mathbb{R}^n$, endowed with the topology of the torus. We remind the reader of the following well-known notation in the field of unfolding: For $x\in \mathbb{R}^n$, the representation $x=\varepsilon\left[\frac{x}{\varepsilon} \right] + \varepsilon \left\{\frac{x}{\varepsilon} \right\}$ holds, where for $z\in \R^n$ the quantity $[z]=(b_1,\dots,b_n)\in \mathbb{Z}^n$ is the unique vector such that $\{ z\}:=z-[z]\in Y$. For a function $f:\mathbb{R}^n\longrightarrow \mathbb{R}$, the unfolded function $\mathcal{T}^{\varepsilon}:\mathbb{R}^n \times Y \longrightarrow \mathbb{R}$ is defined by
$\mathcal{T}^{\varepsilon}(f)(x,y)=f(\varepsilon \left[\frac{x}{\varepsilon} \right] + \varepsilon y)$.

In the sequel, we recall the basic notions from \cite{dobo_constrpumf}.  

\subsection{Unfolding Operators and Integral Identities}


\begin{definition}
We say that an object is $\eps_\mathscr{A}$-periodic, if it is $Y$-periodic in $\R^n$ after transformation with a chart $\phi$ from a designated atlas $\mathscr{A}$.
\end{definition}

For example, if we take a smooth $\eps Y$-periodic function $f:Y \longrightarrow \R$ and a $\phi_\alpha \in \mathscr{A}$, then $\tilde{f}:= f \circ \phi_\alpha= \phi_\alpha^* f$ is $\eps_\mathscr{A}$-periodic on $U_\alpha$. One can also think of $M$ itself being $\eps_\alpha$-periodic, if we image $M$ to represent a material body whose properties (for example heat conductivity etc.) vary in an $\eps_\mathscr{A}$-periodic way. We need the following compatibility condition:

\begin{definition}[UC-criterion]
\label{def:RMuc}
The atlas $\mathscr{A}$ is said to be \emph{compatible with unfolding} (UC) if for all $\alpha, \beta \in I$ with $U_\alpha \cap U_\beta \not= \emptyset$ and for all $\eps $ there exists a $k(\eps) \in \Z^n$ such that
$\phi_\alpha = \phi_\beta + \eps\sum_{i=1}^n k_i(\eps) e_i$ in   $U_\alpha \cap U_\beta$, 
where $e_i$ denotes the $i$-th unit vector in $\R^n$.
\end{definition}
This definition implies that it is a necessary condition for $M$ to be parallelizable. Examples for manifolds which fulfill this criterion include the unit sphere in $\R^2$ and a spherical segment in $\R^3$. 
We can now define unfolding operators locally:

\begin{definition}
\label{def:RMoperator}
Choose a chart $\phi \in \mathscr{A}$ with corresponding domain $U \subset M$.
\begin{enumerate}
\item For a function $f:U\longrightarrow \R$ we define
$\Te_\phi(f)=(\phi \times \Id)^*\, \Te(\phi_* f)$, 
where $\Te$ denotes the usual unfolding operator in $\R^n$. Obviously $\Te_{\Id}= \Te$.

\item For a vector field $F \in \mathfrak{X}(U)$ define analogously
$\Te_\phi(F)=(\phi \times \Id)^*\, \Te(\phi_* F)$. 

\end{enumerate}
\end{definition}

With the help of the partition of unity, we construct a global operator as follows:
\begin{definition}
The global unfolding operator $\Te_\mathscr{A}$ with respect to the atlas $\mathscr{A}$ is defined as
\[
\Te_\mathscr{A}(\cdot) = \sum_{\alpha\in I}\pi_\alpha \Te_{\phi_\alpha}(\cdot |_{U_\alpha}).
\]
\end{definition}
Equivalent definitions can be found in the references cited above.
The next result shows that the unfolding operators are well defined on sets where two charts overlap, giving the well-definedness of $\Te_\mathscr{A}$ on $M$:
\begin{proposition}
\label{prop:RMcompatib}
Let $\phi_\alpha$ and $\phi_\beta$ be two charts of an UC-atlas $\mathscr{A}$ with $U_\alpha \cap U_\beta \not= \emptyset$. Then
$\Te_{\phi_\alpha}= \Te_{\phi_\beta} \quad \text{on } U_\alpha \cap U_\beta$.
\end{proposition}

For the unfolding of integrals, we use the following notation:
\begin{definition}
A sequence $\{f^\eps\}$ in $L^1(M)$ is said to fulfill the \emph{unfolding criterion on manifolds} (UCM) if there exists a function $r:\R^+ \longrightarrow \R$ such that $r(\eps) \longrightarrow 0$ as $\eps \rightarrow 0$ and
\[
\int_M f^\eps \udvol = \frac{1}{|Y|} \int_{M\times Y} \Te_\mathscr{A}(f^\eps) \Te_\mathscr{A}(\sqrt{|G|}) \ud y \ud x + r(\eps).
\]
We write in this case
$\int_M f^\eps \udvol \simeq \frac{1}{|Y|} \int_{M\times Y} \Te_\mathscr{A}(f^\eps) \Te_\mathscr{A}(\sqrt{|G|}) \ud y \ud x$.
\end{definition}
In \cite{dobo_constrpumf} it is shown that the following functions $f$ and sequences $(f^\eps)$ fulfill the (UCM)-criterion:
\begin{enumerate}
\item $f\in L^1(M)$
\item $\{ f^\eps\}\subset L^1(M)$ such that $\nx{f^\eps}{L^1(M)}$ is bounded independently of $\eps$.
\item Since the functions are defined on a compact manifold, the same holds true if we replace $L^1(M)$ with $L^p(M)$ with $1\leq p \leq \infty$.
\end{enumerate}


\begin{proposition}
\label{prop:RMcont}
The operators
\[
\Te_\mathscr{A}:L^p(M) \longrightarrow L^p(M\times Y)
\]
are linear and continuous with operator norm less than $(C_Y(1+ \mathrm{card}(I)\delta))|Y| )^\frac{1}{p}$, where $\delta>0$ is arbitrary,  $\eps \leq \eps_0(\delta)$ and $C_Y$ denotes a constant depending on the coefficients of $g_M$.
\end{proposition}

\subsection{Unfolding of Gradients}

Since gradients are vector fields strongly connected to the Riemannian metric, we need the following

\begin{definition}
For fixed $x\in M$, $\eps>0$ associate to the Riemannian metric $g_M$ on $M$ a $(x,\eps)$-dependent metric $g_Y^{(x,\eps)}$ on $Y$ via $g_Y^{(x,\eps)}(x,\cdot)= \Te_\mathscr{A}(g_M)(x,\cdot)$ in the sense that
\[
\Te_\mathscr{A}(g_M) = \Te_\mathscr{A}( \sum_{i,j} g_{ij}\, \d x^i \otimes \d x^j   ) = \sum_{i,j} \Te_\mathscr{A}(g_{ij}) \,\d y^i \otimes \d y^j.
\]
\end{definition}

\begin{proposition}
\label{prop:RMrm}
Let $F,G \in \mathfrak{X}(M)$. Then
\begin{equation*}
\Te_\mathscr{A}(g_M (F,G) )(x,y) = g_Y^{(x,\eps)}( \Te_\mathscr{A}(F)(x,y), \Te_\mathscr{A}(G)(x,y)   ).
\end{equation*}
\end{proposition}

We need three different gradient operators, which we denote by $\grad_M$, $\grad_Y^{(x, \eps)}$ and $\grad_Y^{(x)}$, resp.\ (we will use the same notation for the corresponding divergence operators): 
\begin{itemize}
\item $\grad_M$ denotes the gradient on $M$ with respect to the metric $g_M$.
\item For fixed $x\in M$ and $\eps>0$, $\grad_Y^{(x,\eps)}$ denotes the gradient on $Y$ with respect to the parameter-dependent metric $g_Y^{(x,\eps)}$.
\item Finally, for fixed $x\in M$ the operator $\grad_Y^{(x)}$ is defined to be the gradient on $Y$ with respect to the parameter-dependent metric $g_Y^{(x)}:=\lim_{\eps \rightarrow 0} \Te_\mathscr{A}(g_M)(x,\cdot)$, i.e.\ with respect to the metric on $Y$ with metric coefficients $g_Y^{(x)}(\frac{\p}{\p y^i}, \frac{\p}{\p y^j})= g_{ij}(x)$.
\end{itemize}

\begin{proposition}
\label{prop:RMgradient}
Let $f:M\longrightarrow \R^n$ be a differentiable function, and let $F\in \mathfrak{X}(M)$ be a differentiable vector field. Then the identities
\begin{gather*}
\eps \Te_\mathscr{A}(\grad_M f)(x,y) = \grad_Y^{(x,\eps)} \Te_\mathscr{A}(f)(x,y) \\
\eps \Te_\mathscr{A}(\div_M F)(x,y) = \div_Y^{(x,\eps)} \Te_\mathscr{A}(F)(x,y)
\end{gather*}
hold in $M\times Y$.
\end{proposition}


\subsection{Compactness Results for Gradients}

One has the following compactness result for gradients:
\begin{theorem}
\label{thm:compactgrad}
Let $\{ w^\eps \} \subset H^1(M)$ be a sequence.
\begin{enumerate}
\item Assume that the estimate
$\nx{w^\eps}{L^2(M)} \leq C$ and 
$\eps\nx{\grad_M w^\eps}{L^2TM} \leq C$
holds with a constant $C$ independent of $\eps$. Then there exists a $w\in L^2(M;H^1_\#(Y))$ such that along a subsequence
\begin{align*}
\Te_\mathscr{A}(w^\eps) &\weak w \quad \text{in } L^2(M\times Y) \\
\eps \Te_\mathscr{A}(\grad_M w^\eps)& \weak \grad_Y^{(x)}w\quad \text{in } L^2(M; L^2TY).
\end{align*}

\item If the stronger estimate
$\nx{w^\eps}{L^2(M)} \leq C$ and
$\nx{\grad_M w^\eps}{L^2TM}  \leq C$ 
holds with a constant $C$ independent of $\eps$, then there exists a $w\in H^1(M)$ and a $\hat{w}\in L^2(M;H^1_\#(Y))$ such that along a subsequence
\begin{align*}
\Te_\mathscr{A}(w^\eps) &\longrightarrow w \quad \text{in } L^2(M\times Y) \\
\Te_\mathscr{A}(\grad_M w^\eps)& \weak (\grad_M w)_Y + \grad_Y^{(x)}\hat{w}\quad \text{in } L^2(M;L^2TY).
\end{align*}
\end{enumerate}
By abuse of notation we use $\grad_Y^{(x)}w$ to denote a function $(x,y) \mapsto \grad_Y^{(x)}w(x,y)$. The operator $(\cdot)_Y$ is defined below.
\end{theorem}

\begin{definition}
\begin{enumerate}
\item For a vector field $F \in \mathfrak{X}(M)$ we define a \emph{transport operator} $(\cdot)_Y$ with
\begin{align*}
(\cdot)_Y:  \mathfrak{X}(M) &\longrightarrow \mathfrak{X}(Y)^M \\
F &\longmapsto F_Y,
\end{align*}
where for $F= \sum_i F^i \frac{\p}{\p x^i}$ the field $F_Y$ is defined via
$F_Y(x,y)= \sum_i F^i(x) \frac{\p}{\p y^i}$.
\item Analogously, we construct a transport operator which maps vector fields on $Y$ to vector fields on $M$:
For a vector field $G \in \mathfrak{X}(Y)$ we define a \emph{transport operator} $(\cdot)_M$ with
\begin{align*}
(\cdot)_M:  \mathfrak{X}(Y) &\longrightarrow \mathfrak{X}(M)^Y \\
G &\longmapsto G_M,
\end{align*}
where for $G= \sum_i G^i \frac{\p}{\p y^i}$ the field $G_M$ is defined via
$G_M(x,y)= \sum_i G^i(y) \frac{\p}{\p x^i}$.
\end{enumerate}
\end{definition}
One easily shows that for vector fields $V_i \in \mathfrak{X}(M)^Y$, $W_i\in \mathfrak{X}(Y)^M$, $i=1,2$, the relations
\begin{itemize}
\item $\Bigl( (V_1)_Y\Bigr)_M=V_1$, $\Bigl( (W_1)_M\Bigr)_Y=W_1$
\item $g_Y^{(x)}(W_1, W_2) = g_M\Bigl( (W_1)_M, (W_2)_M  \Bigr)$
\item $g_M( V_1,V_2 ) = g_Y^{(x)}\Bigl(  (V_1)_Y, (V_2)_Y  \Bigr)$
\end{itemize}
hold.

\section{Application to an Elliptic Model Problem}
\label{sec:applell}

Let $D\in \C_\#(Y)$ be a fixed periodic function such that $0< d_0 \leq D \leq D_0$ for some positive constants $d_0,D_0$. For $\eps>0$ we define the function $D^\eps  :M \longrightarrow \R$ via $D^\eps(x) := D \left(  \left\{  \frac{\phi_\alpha(x)}{\eps} \right\}\right)$ for  $x\in U_\alpha$. Due to the UC-condition we obtain $D \left(  \left\{  \frac{\phi_\alpha(x)}{\eps} \right\}\right) = D \left(  \left\{  \frac{\phi_\beta(x)}{\eps} \right\}\right)$ for $x\in U_\alpha \cap U_\beta$, thus the function $D^\eps$ is well-defined. $D^\eps$ might be interpreted as heat conductivity or diffusivity of $M$ for fixed $\eps >0$.

 Let $c\geq 0$ be a constant and let $f\in L^2(M)$ be a source term. We are considering the problem: Find $\ue \in H^1_0(M)$ with
\begin{subequations}
\label{eq:RMmainstr}
\begin{gather}
-\div_M(\De \grad_M \ue) + c \ue = f \quad \text{in } M \label{eq:RMmainstr1}\\
\ue=0 \quad \text{on } \p M. \label{eq:RMmainstr2}
\end{gather}
\end{subequations}
The weak formulation of this problem reads
\begin{gather}
\int_M^{\phantom{g }} \De g_M(\grad_M \ue ,\grad_M \varphi) \udvol + \int_M c \ue \varphi \udvol = \int_M f \varphi \udvol \quad \forall \varphi \in H^1_0(M). \label{eq:RMweak}
\end{gather}
Formally, the weak formulation is obtained by multiplication of equation \eqref{eq:RMmainstr1} with a test function $\varphi$ and subsequent integration by parts, taking into account the boundary condition \eqref{eq:RMmainstr2}. Existence of a solution for fixed $\eps>0$ is obtained easily by using the Lax-Milgram lemma.

We are going to show the following theorem:
\begin{theorem}
\label{thm:RMhomogeq}
There exists a unique $u\in H^1_0(M)$ such that $\ue\weak u$ in $H^1_0(M)$. The limit satisfies the homogenized equation
\begin{equation}
\begin{gathered}
-\div_M( B \grad_M u  ) + cu =f  \quad \text{in }M\\
u=0 \quad \text{on } \p M,
\end{gathered} \label{eq:RMhomogprob}
\end{equation}
where the linear operator $B$ is constructed with the help of the following parameter-dependent cell problem: For fixed $x\in M$ and $i=1,\dots,n$, find $w_i(x) \in H^1_\#(Y) / \R$, solution of
 \begin{gather*}
-\div_Y^{(x)}( D(\cdot)\grad_Y^{(x)} w_i(x)  ) = \div_Y^{(x)}(D(\cdot) \frac{\p}{\p y^i}) \quad \text{in } Y \\
y \longmapsto w_i(x,y) \quad \text{is $Y$-periodic}.
\end{gather*}
Then define a tensor $A$ as 
$A^k_i(x,y):= \sum_j g^{kj}(x) \frac{\p w_i}{\p y^j}(x,y)$,
and the linear operator B as 
\begin{gather*}
B^k_i(x):= \int_Y^{\phantom{g}} D(y)(\delta^k_i +  A^k_i (x,y)) \ud y.
\end{gather*}
Moreover, the corresponding tensor $\tilde{B}$ with lowered index, i.e.\ $\tilde{B}_{ki}:= \sum_j g_{kj}B^j_i$ is symmetric and positive definite.
\end{theorem}

For the special case of $M$ being a domain in $\R^n$ and $g$ being the Euclidean metric $g_{ij}=\delta_{ij}$, one obtains the operators $\div_M = \div$ (the usual divergence) and $\div^{(x,\eps)}_Y=\div^{(x)}_Y= \div_y$ (the divergence taken with respect to the second variable)  as well as $\grad_M=\grad$ (the usual gradient)  and $\grad_Y^{(x,\eps)}  = \grad_Y^{(x)} = \grad_y$ (the gradient taken with respect to the second variable). Keeping in mind the indentification $\frac{\p}{\p y^i}=e_i$ (the $i$-th unit vecor), one sees that the theorem above generalizes the well-known homogenization results for a strongly elliptic equation, see e.g.\ \cite{cido_intrhom}.

\subsection{A-priori Estimates and Limits}
\label{sec:RMestimlim}

Using $\ue$ as a test function in the weak formulation, one shows that the estimate $\nx{\ue}{H^1(M)} \leq C$ holds with a constant $C>0$ independent of $\eps$.

Theorem \ref{thm:compactgrad} and the usual compactness results and embeddings (see e.g.\ \cite{nonlinanamanif}) now show that there exits a $u\in H^1(M)$ and a $\hat{u}\in L^2(M;H^1_\#(Y))$ such that along a subsequence
\begin{subequations}
\label{eq:convergence}
\begin{align}
\ue& \weak u \quad \text{in } H^1(M) \\
\ue &\longrightarrow u \quad \text{in } L^2(M) \\
\Te_\mathscr{A}(u^\eps) &\longrightarrow u \quad \text{in } L^2(M\times Y) \\
\Te_\mathscr{A}(\grad_M u^\eps)& \weak (\grad_M u)_Y + \grad_Y^{(x)}\hat{u}\quad \text{in } L^2(M;L^2TY).
\end{align}
\end{subequations}
Due to the compactness of the trace map $H^1(M) \hookrightarrow L^2(\p M)$ we have $u|_{\p M} = \lim_{\eps \rightarrow 0} (\ue|_{\p M}) =0$


\subsection{The Two-Scale Limit Problem}

In order to derive the limit problem, we choose two test functions $\varphi_1 \in \C^\infty_0(M)$ and $\varphi_2\in \C^\infty_0(M;\C^\infty_\#(Y))$ and define
\[
\varphi^\eps(x):= \varphi_1(x) + \eps \varphi_2 \left(x,\left\{ \frac{\phi_\alpha(x)}{\eps} \right\} \right) \quad \text{for } x\in U_\alpha.
\]
We need the following auxiliary result:
\begin{lemma}
\label{lem:RMauxconvergence}
We have
\begin{enumerate}
 \item $\Te_\mathscr{A}(\De)(x,y)=D(y)$
 \item $\Te_\mathscr{A}( \grad_M \varphi^\eps ) \longrightarrow (\grad_M \varphi_1)_Y + \grad_Y^{(x)} \varphi_2$ in $L^\infty(M\times Y)$
 \item $\Te_\mathscr{A}( \varphi^\eps ) \longrightarrow\varphi_1$ in $L^\infty(M\times Y)$
\end{enumerate}
\end{lemma}
\begin{proof}
By the very definition of pullback and pushforward, we obtain for $\Te_{\phi_\alpha}$, $\alpha \in I$
\[
\Te_{\phi_\alpha}(\De)(x,y) = D\left( \left\{ \frac{\eps \left[ \frac{\phi_\alpha(x)}{\eps} \right] + \eps y  }{\eps}  \right\}\right) = D(y).
\]
Next, we have due to the unfolding rules for gradients (see Proposition \ref{prop:RMgradient})
\[
\Te_\mathscr{A}( \grad_M \varphi^\eps ) = \Te_\mathscr{A}(\grad_M \varphi_1) + \grad_Y^{(x,\eps)} \Te_\mathscr{A}(\varphi_2).
\]
For the first term on the right hand side we obtain for $x\in U_\alpha$ the convergence
\begin{align*}
\Te_{\phi_\alpha}(\grad_M \varphi_1)(x,y)& = \Te_{\phi_\alpha}(\sum_{i,j} g^{ij} \frac{\p \varphi_1}{\p x^i} \frac{\p}{\p x^j}  )(x,y) \\
&= \sum_{i,j} \underbrace{\Te_{\phi_\alpha}(g^{ij})(x,y)}_{\rightarrow g^{ij}(x)} \underbrace{\Te_{\phi_\alpha}(\frac{\p \varphi_1}{\p x^i})(x,y)}_{\rightarrow \frac{\p \varphi_1}{\p x^i}} \frac{\p}{\p y^j} \\
& \longrightarrow \sum_{i,j} g^{ij}(x) \frac{\p \varphi_1}{\p x^i} \frac{\p }{\p y ^j} = (\grad_M \varphi_1)_Y
\end{align*}
in $\C(M\times Y)$. For the second term, we have as above 
\begin{align*}
\grad_Y^{(x,\eps)} \Te_{\phi_\alpha}(\varphi_2)(x,y) &= \grad_Y^{(x,\eps)} \varphi_2 (   \phi_\alpha^{-1}( \eps \left[ \frac{\phi_\alpha(x)}{\eps}\right]   + \eps y ), y   ) \\
&=\sum_{i,j} \Te_{\phi_\alpha}(g^{ij})(x,y) \frac{\p \varphi_2 }{\p y^i}(   \phi_\alpha^{-1}( \eps \left[ \frac{\phi_\alpha(x)}{\eps}\right]   + \eps y ), y   ) \frac{\p}{\p y^j},
\end{align*}
where $\frac{\p \varphi_2 }{\p y^i}$ has to be understood as derivative with respect to the second variable. Since $\Te_{\phi_\alpha}(g^{ij}) \rightarrow g^{ij}$ as well as 
\begin{equation}
\phi_\alpha^{-1}( \eps \left[ \frac{\phi_\alpha(x)}{\eps}\right]   + \eps y)\longrightarrow \phi_\alpha^{-1}( \phi_\alpha(x) )=x \label{eq:RMconvchart}
\end{equation}
(due to the continuity of $\phi_\alpha$), we obtain by using the continuity of $\frac{\p \varphi_2 }{\p y^i}$ that 
\begin{align*}
\grad_Y^{(x,\eps)} \Te_{\phi_\alpha}(\varphi_2)(x,y) & \longrightarrow \sum_{i,j} g^{ij}(x) \frac{\p \varphi_2}{\p y^i}(x,y)\frac{\p}{\p y^j} = \grad_Y^{(x)} \varphi_2(x,y).
\end{align*}
The last assertion follows along the same lines by using the boundedness of $\varphi_2$ as well as \eqref{eq:RMconvchart}.
\qquad\end{proof}

We choose $\varphi=\varphi^\eps$ as a test function in the weak formulation \eqref{eq:RMweak}. Since all the terms appearing under the integrals in \eqref{eq:RMweak} are bounded in $L^1(M)$ independent of $\eps$, these terms satisfy the UCM-criterion, and we can unfold the integral identity with respect to~$\simeq$. 
We obtain the expression 
\begin{align*}
\frac{1}{|Y|} \int_{M\times Y}^{\phantom{g}} & D(y) g_Y^{(x,\eps)}\Bigl(  \Te_\mathscr{A}(\grad_M \ue) , \Te_\mathscr{A}(\grad_M \varphi^\eps) \Bigr)(x,y)\,   \Te_\mathscr{A}(\sqrt{|G|})(x,y)\ud y \ud x \\
&+ \frac{1}{|Y|} \int_{M\times Y}  c \Te_\mathscr{A}(\ue)(x,y) \Te_\mathscr{A}(\varphi^\eps)(x,y) \,   \Te_\mathscr{A}(\sqrt{|G|})(x,y)\ud y \ud x \\
&= \frac{1}{|Y|} \int_{M\times Y}  \Te_\mathscr{A}(f)(x,y)  \Te_\mathscr{A}(\varphi^\eps)(x,y)\,   \Te_\mathscr{A}(\sqrt{|G|})(x,y)\ud y \ud x  + r(\eps)
\end{align*}
with $r(\eps)\rightarrow 0$ as $\eps \rightarrow 0$. Taking the limit on both sides, keeping in mind the convergences \eqref{eq:convergence}, one obtains the \emph{two-scale formulation of the limit problem}
\begin{equation}
\begin{gathered}
\frac{1}{|Y|} \int_{M\times Y}^{\phantom{g}}  D(y) g_Y^{(x)}\Bigl(  (\grad_M u)_Y + \grad_Y^{(x)}\hat{u} , (\grad_M \varphi_1)_Y + \grad_Y^{(x)}\varphi_2 \Bigr)  \ud y \udvol \\
+ \int_{M}  c u \varphi_1\ud y \udvol = \int_{M}  f \varphi_1\ud y \udvol .
\end{gathered} \label{eq:RMweaklim}
\end{equation}
By density of test functions, this holds for all $(\varphi_1, \varphi_2) \in H^1_0(M)\times L^2(M;H^1_\#(Y))$. The strong formulation of this problem is given in Theorem \ref{thm:RMhomogeq}.

\subsection{Proof of Theorem \ref{thm:RMhomogeq}}

\textbf{Step 1} The cell problem:\\
We start with the weak formulation \eqref{eq:RMweaklim}: Choosing $\varphi_1=0$, one obtains $\frac{1}{|Y|} \int_{M\times Y}  D(y) g_Y^{(x)}\Bigl(  (\grad_M u)_Y + \grad_Y^{(x)}\hat{u} , \grad_Y^{(x)}\varphi_2 \Bigr)  \ud y \udvol =0$. Upon integration by parts, this yields
\[
- \int_{M\times Y} \div_Y^{(x)} \Bigl( D \bigl[ (\grad_M u)_Y + \grad_Y^{(x)}\hat{u}   \bigr]  \Bigr) \varphi_2 \ud y \udvol = 0 \qquad \forall \varphi_2 \in L^2(M;H^1_\#(Y)),
\]
the strong form of which is given by: For fixed $x\in M$, find $\hat{u}(x) \in H^1_\#(Y)/\R$ such that
\begin{equation}
\begin{split}
-\div_Y^{(x)}(  D(\cdot)  \grad_Y^{(x)}\hat{u}(x)(y))& = \div_Y^{(x)}(D(\cdot)  (\grad_M u)_Y(x,y)    ) \quad \text{in }M \\
y&\longmapsto \hat{u}(x)(y) \quad \text{is $Y$-periodic}.
\end{split} \label{eq:RMuhat}
\end{equation}
To "factor out" the term $(\grad_M u)_Y$, we construct a solution of the cell problem for $i=1,\dots,n$, given by: Find a solution $w_i(x)\in H^1_\#(Y)/ \R$ of
\begin{subequations}
\label{eq:cell}
\begin{gather}
-\div_Y^{(x)}( D(y)\grad_Y^{(x)} w_i(x)(y)  ) = \div_Y^{(x)}(D(y) \frac{\p}{\p y^i}) \quad \text{in } Y \\
y \longmapsto w^i(x,y) \quad \text{is $Y$-periodic}.
\end{gather}
\end{subequations}
The weak formulation of this problem
\begin{equation}
\int_Y D(y) g_Y^{(x)}\Bigl( \grad_Y^{(x)} w_i(x), \grad_Y^{(x)} \varphi   \Bigr) \ud y = - \int_Y D(y) g_Y^{(x)}\Bigl( \frac{\p}{\p y^i}, \grad_y^{(x)} \varphi \Bigr) \ud y \quad \forall \varphi\in H^1_\#(Y)/ \R \label{eq:RMcellweak}
\end{equation}
is well defined in the indicated function space and thus 
has a solution $w^i(x)$, which is unique up to constants. Define $\hat{u}(x,y)= \sum_i w_i(x,y) (\grad_M u(x))^i$. The following calculation shows that this $\hat{u}$ is a solution of \eqref{eq:RMuhat}: The periodicity in the variable $y$ is obvious, and we have
\begin{align*}
-\div_Y^{(x)}(  D  \grad_Y^{(x)}\hat{u}) & = -\sum_{i=1}^n (\grad_M u)^i \div_Y^{(x)}( D \grad_Y^{(x)} w_i  ) \\ &= \div_Y^{(x)} \Bigl(  D \sum_{i=1}^n (\grad_M u)^i \frac{\p}{\p y^i} \Bigr) 
 =\div_Y^{(x)}(D (\grad_M u)_Y ).
\end{align*}

\textbf{Step 2} The homogenized problem:\\
We now choose $\varphi_2=0$ in \eqref{eq:RMweaklim} to obtain $\frac{1}{|Y|} \int_{M\times Y} D(y) g_Y^{(x)}\Bigl(  (\grad_M u)_Y + \grad_Y^{(x)}\hat{u} , (\grad_M \varphi_1)_Y \Bigr)  \ud y \udvol + \int_{M}  c u \varphi_1\ud y \udvol = \int_{M}  f \varphi_1\ud y \udvol$. Inserting $\hat{u}$ and using the remarks following the definition of the transport operators, this is equivalent to
\begin{align*}
\frac{1}{|Y|} &\int_{M\times Y} D(y) g_M\Bigl(  \grad_M u +\sum_{i=1}^n  (\grad_M u)^i(\grad_Y^{(x)}  w_i )_M , \grad_M \varphi_1 \Bigr)  \ud y \udvol \\
&+ \int_{M}  c u \varphi_1\ud y \udvol = \int_{M}  f \varphi_1\ud y \udvol.
\end{align*}
Upon an integration by parts, we obtain the following strong form:
\begin{gather*}
-\div_M\Bigl(  \int_Y D[ \grad_M u  +\sum_{i=1}^n  (\grad_M u)^i(\grad_Y^{(x)}  w_i )_M ]   \ud y \Bigr) + cu = f \quad \text{in } M \\
u=0 \quad \text{on } \p M
\end{gather*}
It remains to characterize the expression 
\[K(x,y):= D(y)[ \grad_M u(x)  +\sum_{i=1}^n  (\grad_M u(x))^i(\grad_Y^{(x)}  w_i(x,y) )_M] .\] Written component-wise, we obtain
\begin{align*}
K(x,y)&= \sum_k D(y)\Bigl(  (\grad_M u(x))^k + \sum_{i,j} (\grad_M u(x))^i g^{kj}(x) \frac{\p w_i(x,y)}{\p y^j} \Bigr) \frac{\p}{\p x^k} \\
& = \sum_{k,i,j} D(y)\Bigl(  \delta^k_i(\grad_M u(x))^i + (\grad_M u(x))^i g^{kj}(x) \frac{\p w_i(x,y)}{\p y^j} \Bigr) \frac{\p}{\p x^k} \\
& = \sum_{k,i} D(y)\Bigl(  \delta^k_i + \sum_j g^{kj}(x) \frac{\p w_i(x,y)}{\p y^j} \Bigr) (\grad_M u(x))^i \frac{\p}{\p x^k} .
\end{align*}
Another expression for $K$ is given by \[K(x,y)= \sum_{k,i} D(y)(\delta^k_i + (\grad_Y^{(x)} w_i(x,y) )^k   ) (\grad_M u(x))^i \frac{\p}{\p x^k}.\]

The part $x\mapsto\delta^k_i + \sum_j g^{kj}(x) \frac{\p w_i(x,y)}{\p y^j}$ corresponds to a linear map in tensorial notation. We set $A(x,y):=[\sum_j g^{kj}(x) \frac{\p w_i(x,y)}{\p y^j}]^k_i$. Since $\Id=[\delta^k_i]^k_i$, we can apply $(\Id +A(\cdot,y))$ to $\grad_M u$ to obtain
\[
K(\cdot,y)=D(y)(\Id +A)(\grad_M u).
\] 
Integrating over $Y$, we get the expression $B\grad_M u$ as stated in the theorem. Note however that at this point we do not know whether $B$ is a tensor, i.e.\ invariant under coordinate changes. This is due to the fact that the lower index $i$ stems from an index number (of the function $w_i$) and not from a tensorial expression itself. On the other hand, the upper index $k$ stems from a tensorial expression and thus $B$ is contravariant in this index. \\
We overcome this difficulty with the result of step 3: There it is shown that the expression  $\tilde{B}=[\tilde{B}_{ki}]$, corresponding to $B$ with a lowered index $k$, is symmetric. Since $B$ is contravariant in $k$, $\tilde{B}$ is covariant in $k$ and thus, due to the symmetry, also in $i$. Therefore $B$ has to be covariant in $i$ as well, and $B$ is finally a well-defined mixed tensor corresponding to a linear map acting on vector fields.

\textbf{Step 3} Properties of the homogenized linear operator:\\
Define $\tilde{B}_{ki}=  \sum_j  g_{kj} B^j_i$. In order to show that $\tilde{B}$ is symmetric, we start with the weak formulation of the cell problem \eqref{eq:RMcellweak} for $i=\alpha$, where we use $\varphi=w_\beta$ as a test function ($\alpha, \beta \in \{ 1,\dots,n \}$),
\begin{equation*}
\int_Y D(y) g_Y^{(x)}\Bigl( \grad_Y^{(x)} w_\alpha, \grad_Y^{(x)} w_\beta   \Bigr) \ud y = - \int_Y D(y) g_Y^{(x)}\Bigl( \frac{\p}{\p y^\alpha}, \grad_y^{(x)} w_\beta \Bigr) \ud y. 
\end{equation*}
In component notation, this reads
\[
\sum_{i,j}  \int_Y D(y) g_{ij}(\grad_Y^{(x)} w_\alpha)^i (\grad_Y^{(x)} w_\beta)^j \ud y= - \sum_{j}\int_Y D(y) g_{\alpha j} (\grad_Y^{(x)} w_\beta)^j \ud y.
\]
Since $g_{\alpha j}= \sum_i \delta^i_\alpha g_{ij}$, above expression is equivalent to
\begin{equation}
\sum_{i,j} \int_Y D(y)g_{ij}[  (\grad_Y^{(x)} w_\alpha)^i + \delta^i_\alpha  ] (\grad_Y^{(x)} w_\beta)^j \ud y =0. \label{eq:RMsymm}
\end{equation}
Now we have
\begin{align*}
\tilde{B}_{\beta\alpha} &= \sum_i g_{\beta i} [ \int_Y D(y)(\Id + A(\cdot,y))  \ud y]_\alpha^i  = \sum_i \int_Y D(y) g_{\beta i} \Bigl(  \delta^i_\alpha + \sum_{k} g^{ik} \frac{\p w_\alpha}{\p y^k} \Bigr) \ud y \\
& = \sum_i \int_Y D(y) g_{\beta i} \Bigl(  \delta^i_\alpha + ( \grad_Y^{(x)} w_\alpha  )^i \Bigr) \ud y  = \sum_{i,j}  \int_Y D(y) \delta^j_\beta  g_{ij} \Bigl(  \delta^i_\alpha + ( \grad_Y^{(x)} w_\alpha  )^i \Bigr) \ud y
\end{align*}
Adding the expression \eqref{eq:RMsymm}, we get
\[
\tilde{B}_{\beta\alpha} = \sum_{i,j}  \int_Y D(y)  g_{ij}\Bigl( \delta^j_\beta + (\grad_Y^{(x)}  w_\beta)^j \Bigr) \Bigl(  \delta^i_\alpha + ( \grad_Y^{(x)} w_\alpha  )^i \Bigr) \ud y.
\]
We easily see that $\tilde{B}_{\beta\alpha}=\tilde{B}_{\alpha\beta} $. Since $\alpha,\beta\in \{ 1,\dots, n \}$ is arbitrary, $\tilde{B}$ is symmetric.

Next, we show that $\tilde{B}$ is positive: To this end, let $V$ be a vector field on $M$. Then
\begin{align}
\sum_{\alpha, \beta} & \tilde{B}_{\alpha \beta} V^\alpha V^\beta = \sum_{i,j, \alpha, \beta}  \int_Y D(y)  g_{ij}\Bigl( \delta^j_\beta + (\grad_Y^{(x)}  w_\beta)^j \Bigr) \Bigl(  \delta^i_\alpha + ( \grad_Y^{(x)} w_\alpha  )^i \Bigr)V^\alpha V^\beta \ud y \notag\\
&  = \sum_{i,j}  \int_Y D(y)  g_{ij} \sum_\beta V^\beta \Bigl( \delta^j_\beta + (\grad_Y^{(x)}  w_\beta)^j \Bigr) \sum_{\alpha} V^\alpha \Bigl(  \delta^i_\alpha + ( \grad_Y^{(x)} w_\alpha  )^i \Bigr) \ud y  \label{eq:RMdefin} \\
& = \sum_{i,j} \int_Y D(y)g_{ij}\zeta^i \zeta^j \ud y \geq 0,  \notag
\end{align}
since $D$ is positive and the $g_{ij}$'s are the coefficients of a Riemannian metric. Here $\zeta^i = \sum_\beta V^\beta \Bigl( \delta^j_\beta + (\grad_Y^{(x)}  w_\beta)^j \Bigr) $.

We now show that $\tilde{B}$ is definite. Let $V$ again be a vector field on $M$ and assume that $V\not=0$. Let $\sum_{\alpha,\beta}  \tilde{B}_{\alpha \beta} V^\alpha V^\beta =0$. Keeping in mind the definition of $\tilde{B}$, the index-free version of the second line of the previous considerations \eqref{eq:RMdefin} reads as
\[
\sum_{\alpha, \beta}  \tilde{B}_{\alpha \beta} V^\alpha V^\beta= \int_Y D(y) g_M((\Id + A)V,V) \ud y = \int_Y D(y) g_M((\Id + A)V,(\Id + A)V) \ud y.
\]
The assumption $\sum_{\alpha,\beta}  \tilde{B}_{\alpha \beta} V^\alpha V^\beta =0$ is equivalent to $g_M((\Id + A)V,(\Id + A)V)=0$ due to the positivity of $D$. Using the transport operator, this is equivalent to $g_Y^{(x)}\Bigl( (\widetilde{\Id +A})(V)_Y, (\widetilde{\Id +A})(V)_Y\Bigr)=0 $, which in turn is equivalent to $(\widetilde{\Id +A})(V)_Y=0$. Here $(\widetilde{\Id +A})$ is a map acting on a parameter-dependent vector field $W$ on $Y$ via
\[
\sum_i W^i(x,y) \frac{\p}{\p y^i} \longmapsto \sum_{i,j}[\delta^j_i  + (\grad_Y^{(x)} w_i(x,y) )^j ]W^i(x,y) \frac{\p}{\p y^j}.
\]
Consider for $i=1,\dots,n$ the auxiliary functions $\eta^i:(Y\longrightarrow \R)^M$ given by $\eta^i(x,y)= \sum_j g_{ij}(x) y^i$. It holds that
\begin{align*}
\grad_Y^{(x)} \eta^i &= \sum_{k,l} g^{kl} \frac{\p \eta^i}{\p y^l} \frac{\p}{\p y^k} = \sum_{k,l} \sum_j g^{kl} g_{ij} \underbrace{\frac{\p y^j}{\p y^l}}_{=\delta^j_k} \frac{\p}{\p y^k} \\
& = \sum_{k,l} \underbrace{g^{kl} g_{il}}_{=\delta^i_k} \frac{\p}{\p y^k} = \frac{\p}{\p y^i}.  
\end{align*}
Note that $\eta^i$ corresponds to the function $y\mapsto y^i$ in the corresponding proof from homogenization in $\R^n$, and $\frac{\p}{\p y^i}  $ corresponds to the unit vector $e_i$.

With the help of this auxiliary function $\eta^i$, we obtain
\begin{align*}
0= (\widetilde{\Id +A})(V)_Y = \sum_i \grad_Y^{(x)}( \eta^i - w^i  ) V^i =  \sum_i \grad_Y^{(x)}[ (\eta^i - w^i)V^i  ] ,
\end{align*}
since $V^i$ depends only on $x\in M$ and not on $y\in Y$. Therefore $\sum_i (\eta^i - w^i)V^i = \text{const.}$, with a constant depending on $x$ but not on $y$. This amounts to say that
\begin{equation}
\sum_i \eta^i(x,y) V^i(x) - \text{const}(x) = \sum_i w^i(x,y) V^i(x). \label{eq:RMcontrad}
\end{equation}
Since $V\not=0$, there exists a $x\in M$ with $V(x) \not=0$. Then (using matrix notation)
\begin{align*}
\sum_i \eta^i(x,y)V^i(x)& = \sum_{i,j} g_{ij}(x)y^jV^i(x) \\
& =\left( G(x) \begin{pmatrix} V^1(x) \\ \vdots \\V^n(x)     \end{pmatrix}\right)^T \cdot \begin{pmatrix} y^1 \\ \vdots \\ y^n     \end{pmatrix}.
\end{align*}
Since $\left(\begin{smallmatrix} V^1(x)\\\vdots \\V^n(x)  \end{smallmatrix}\right)\not=0$ and $G(x)$ is invertible,  $G(x) \left(\begin{smallmatrix} V^1(x) \\ \vdots \\V^n(x)     \end{smallmatrix}\right)$ is not equal to $0$ as well and thus
\[
\left( G(x) \begin{pmatrix} V^1(x) \\ \vdots \\V^n(x)     \end{pmatrix}\right)^T \cdot \begin{pmatrix} y^1 \\ \vdots \\ y^n     \end{pmatrix} \not=0
\]
for some choice of $y$. Especially, this expression is not $Y$-periodic in $y$. However, the right hand side of \eqref{eq:RMcontrad} is periodic in $y$. Thus we have reached a contradiction. This shows that $\sum_{\alpha,\beta}  \tilde{B}_{\alpha \beta} V^\alpha V^\beta =0$ implies $V=0$ and finishes the proof of the theorem.

Why is the matrix $\tilde{B}$ so important? This is due to the fact that it appears naturally in the weak formulation of the homogenized problem: Upon multiplication with a test function $\varphi\in H^1_0(M)$ and integration by parts, problem \eqref{eq:RMhomogprob} reads as
\begin{equation}
\int_M g_M(B\grad_M u, \grad_M \varphi) \udvol + \int_M cu\varphi \udvol = \int_M f \varphi \udvol. \label{eq:weakhomog}
\end{equation}
The first term can now be written in component notation as
\begin{align*}
g_M(B\grad_M u, \grad_M \varphi) &= \sum_{i,j} \sum_{\alpha} g_{ij} B^j_\alpha (\grad_M u)^\alpha (\grad_M \varphi)^i \\
&= \sum_{i,\alpha} \tilde{B}_{i\alpha} (\grad_M u)^\alpha (\grad_M \varphi)^i=:g_B(\grad_M u, \grad_M \varphi).
\end{align*}
Due to the properties of $\tilde{B}$, $g_B$ is a symmetric and coercive bilinear form on $M$, and the lemma of Lax-Milgram can be applied to the weak formulation \eqref{eq:weakhomog} above to obtain the existence and uniqueness of a solution $u$.

As a corollary to the fact that the solution of the homogenized problem is unique, we note:
\begin{corollary}
The convergence in Theorem \ref{thm:RMhomogeq} holds for the whole sequence $\ue$ (and not only for a subsequence).
\end{corollary}

With the help of the implicit function theorem for Banach spaces, one shows that the cell problem is smooth with respect to the parameter:
\begin{theorem}
For the solution $w_i$ of the cell problem \eqref{eq:cell} it holds
\begin{equation*}
w_i \in \Omega^1_0( M, H^1_\#(Y)/\R    ), \quad i=1\dots,n.
\end{equation*}
\end{theorem}
Here $\O_l^k(M,X)$ denotes the space of $k$-times continuously differentiable $l$-forms in $M$ with values in the Banach space $X$, see \cite{cartandifforms}.

\section{Equivalent Atlases}
\label{sec:RMequivatl}

In this section we construct an equivalence relation between certain UC-atlases for $M$ and show that all equivalent atlases lead to the same limit problem \eqref{eq:RMhomogprob} in Theorem \ref{thm:RMhomogeq}. 
To this end, let $Y$ and $Z$ be two rectangular subsets of $\R^n$ (not necessarily restricted to $[0,1)^n$), representing two reference cells. We assume both cells to be equipped with the chart $\Id$ and denote the local basis vectors by $\frac{\p}{\p y^i}$ (for $Y$) and $\frac{\p}{\p z^i}$ (for $Z$), $i=1,\dots, n$. The corresponding dual forms are denoted by $\d y^i$ and $\d z^i$, respectively. 

Let there be given two scalar functions $D_Y:Y\longrightarrow \R$ and $D_Z:Z \longrightarrow \R$, representing e.g.\ material properties as above. We need the following assumptions:
\begin{ass}
\label{ass:RMequiv}
Let the manifold $M$ be equipped with two atlases $\mathscr{A}^1=\{ \phi^1_\alpha: U^1_\alpha \longrightarrow V^1_\alpha; \alpha \in I   \}$ and $\mathscr{A}^2=\{ \phi^2_\alpha: U^2_\alpha \longrightarrow V^2_\alpha; \alpha \in \tilde{I}   \}$, both satisfying the UC-criterium, with some finite index sets $I$ and $\tilde{I}$. Assume that whenever $U^1_\alpha \cap U^2_\beta \not= \emptyset$ for some $\alpha\in I, \beta \in \tilde{I}$, then the coordinate transformation $\phi^2_\beta \circ (\phi^1_\alpha)^{-1}$ is the restriction of some linear map $\tilde{F}:\R^n \longrightarrow \R^n$ to $V^1_\alpha$. Lemma \ref{lem:RMtransfocell} shows that this map $\tilde{F}$ is unique across different charts, thus it is not restrictive to assume the existence of \emph{one} linear map $F:\R^n \longrightarrow \R^n$ such that $\phi^2_\beta \circ (\phi^1_\alpha)^{-1}=F|_{V^1_\alpha}$ for all suitable index pairs.

Furthermore, assume that $F|_Y:Y\longrightarrow Z$ is a coordinate transformation between the reference cells such that for the functions $D_Y$ and $D_Z$ representing material properties, we have $D_Z= (F|_Y)_* D_Y$.
\end{ass}

To give a simple example how these assumptions apply, consider the following situation: 
\label{exmp:RMsk}
Let $\O\subset \R^2$ be a domain. We equip $\O$ with two different atlases, each consisting of one chart: $\mathscr{A}_1:=\{ \Id:\O \longrightarrow \O \}$ as well as $\mathscr{A}_2:=\{ \mathrm{Sk}:\O\longrightarrow \mathrm{Sk}(\O) \}$, where the map $\mathrm{Sk}$ is given by
\begin{align*}
\mathrm{Sk}:\R^n &\longrightarrow \R^n \\
\begin{pmatrix} x_1 \\ x_2   \end{pmatrix} & \longmapsto \begin{pmatrix} x_2 \\ x_1 \end{pmatrix}.
\end{align*} 
As reference cells, we use $Y_1=[0,1)^2$ (for $\mathscr{A}_1$) and $Y_2=[0,1)^2$ (for $\mathscr{A}_2$). Furthermore, let $D_{Y_1}:Y_1\longrightarrow \R$ be a function (representing material properties in the first reference cell) and set $D_{Y_2}(y_1,y_2):=D_{Y_1}(y_2,y_1)$.
Since $\mathrm{Sk} \circ \Id^{-1}=\mathrm{Sk} :\O \longrightarrow \mathrm{Sk}(\O)  $ can be trivially extended to the linear map $\mathrm{Sk}$, defined on the whole of $\R^n$, and since $\mathrm{Sk}:Y_1 \longrightarrow Y_2$ is a coordinate transformation of the reference cells such that $D_{Y_2} = \mathrm{Sk}_* D_{Y_1}$, the Assumptions \ref{ass:RMequiv} are fulfilled. 

In the situation of the preceding example, one can also consider $\mathscr{A}_1:=\{ \Id:\O \longrightarrow \O \}$ and $\mathscr{A}_2:=\{ 2\Id:\O \longrightarrow 2\Id(\O) \}$, with reference cells $Y_1=[0,1)^2$ as well as $Y_2=[0,2)^2$ and functions $D_{Y_1}$ as above with $D_{Y_2}(y_1,y_2)= D_{Y_1}( \frac{y_1}{2},\frac{y_2}{2} )$. Here
\begin{align*}
2\Id:\R^n & \longrightarrow \R^n \\
\begin{pmatrix} x_1 \\ x_2   \end{pmatrix} & \longmapsto \begin{pmatrix} 2x_1 \\ 2x_2 \end{pmatrix}.
\end{align*}

\subsection{Results}

The main result of this section is the following
\begin{theorem}
\label{thm:independence}
Under the Assumptions \ref{ass:RMequiv}, the limit problem \eqref{eq:RMhomogprob} is independent of the atlas $\mathscr{A}_i$, $i=1,2$.
\end{theorem}

The class of atlases satisfying the assumptions given above constitutes an equivalence relation, as one can verify directly:
\begin{proposition}
Let $\mathscr{A}$ and $\mathscr{B}$ be two atlases for $M$, both satisfying the UC-criterion. We write $\mathscr{A}\sim \mathscr{B}$ to denote that the couple ($\mathscr{A},\mathscr{B})$ satisfies the Assumptions \ref{ass:RMequiv}. Then the relation '$\sim$' is an equivalence relation on the set of UC-atlases.
\end{proposition}
This result tacitly assumes that there exist reference cells $Y_{\mathscr{A}}$ (belonging to $\mathscr{A}$) and $Y_{\mathscr{B}}$ (belonging to ${\mathscr{B}}$) etc.\ as well as "material-property"-functions $D_{Y_{\mathscr{A}}}$ and $D_{Y_{\mathscr{B}}}$ as stated above.

\subsection{Proof of Theorem \ref{thm:independence}}


We begin with showing the asserted uniqueness of the transformation $F$ and collect further results needed for subsequent derivations:

\begin{lemma}
\label{lem:RMFunique}
Let $\alpha, \tilde{\alpha}$ be two indices with $V^1_\alpha \cap V^1_{\tilde{\alpha}}\not= \emptyset$. Choose $\beta, \tilde{\beta}$ such that $U^1_\alpha \cap U^2_\beta\not = \emptyset$ and $U^1_{\tilde{\alpha}} \cap U^2_{\tilde{\beta}} \not= \emptyset$. Then $\phi^2_\beta \circ (\phi^1_\alpha)^{-1} = \phi^2_{\tilde{\beta}}\circ (\phi^1_{\tilde{\alpha}})^{-1}$ in $V^1_\alpha \cap V^1_{\tilde{\alpha}}$. Thus the linear map $\tilde{F}$ is unique for all index pairs.
\end{lemma}
\begin{proof}
Due to the UC-criterion, there exists $K, \tilde{K} \in \R^n$ such that $\phi^1_{\tilde{\alpha}}= \phi^1_\alpha + K$ as well as $\phi^2_{\tilde{\beta}}= \phi^2_\beta + \tilde{K}$. This implies $(\phi^1_{\tilde{\alpha}})^{-1}= (\phi^1_\alpha)^{-1}(\cdot - K)$.

 Since $\phi^2_\beta \circ (\phi^1_\alpha)^{-1}$ as well as $\phi^2_{\tilde\beta} \circ (\phi^1_{\tilde{\alpha}})^{-1}$ are supposed to be linear, it holds $\phi^2_\beta \circ (\phi^1_\alpha)^{-1} = D(\phi^2_\beta \circ (\phi^1_\alpha)^{-1})$ and $\phi^2_{\tilde\beta} \circ (\phi^1_{\tilde{\alpha}})^{-1}= D(\phi^2_{\tilde\beta} \circ (\phi^1_{\tilde{\alpha}})^{-1})$, where $D$ denotes the total derivative. By the chain rule, we obtain
\[
D[\phi^2_{\tilde\beta} \circ (\phi^1_{\tilde{\alpha}})^{-1}] = D[ \phi^2_\beta( (\phi^1_\alpha)^{-1}(\cdot - K)  ) + \tilde{K}  ] = D[\phi^2_\beta \circ (\phi^1_\alpha)^{-1}   ].
\]
This implies the asserted equality and the uniqueness of the linear map $\tilde{F}$.
\qquad\end{proof}

If we interpret \eqref{eq:RMmainstr} as a stationary heat equation, then $D$ stands for a material property -- in this case the heat conductivity of the underlying material.
The next lemma shows that the description of these material properties is independent of the atlas:
\begin{lemma}
\label{lem:RMDeps}
Under the assumptions \ref{ass:RMequiv}, it holds 
\[
D^\eps(x)= D_Y(\frac{\phi^1_\alpha(x)}{\eps})=D_Z(\frac{\phi^2_\beta(x)}{\eps})
\] 
for $x\in U^1_\alpha \cap U^2_\beta$. 
\end{lemma}
\begin{proof}
Since $F=\phi^2_\beta\circ (\phi^1_\alpha)^{-1}$ is linear, one has $\frac{\phi^2_\beta\circ (\phi^1_\alpha)^{-1}(\phi^1_\alpha(x))}{\eps}= \phi^2_\beta\circ (\phi^1_\alpha)^{-1}(\frac{\phi^1_\alpha(x)}{\eps})$. This gives
\begin{align*}
D^\eps(x) &:= D_Z( \frac{\phi^2_\beta(x)}{\eps}) = D_Z(\frac{\phi^2_\beta\circ (\phi^1_\alpha)^{-1}(\phi^1_\alpha(x))}{\eps}    ) = D_Z(\underbrace{\phi^2_\beta\circ (\phi^1_\alpha)^{-1}}_{=F} (\frac{\phi^1_\alpha(x)}{\eps}))\\
&= (F^*D_Z) ( \frac{\phi^1_\alpha(x)}{\eps})  .
\end{align*}
Since $D_Z = F_* D_Y \Leftrightarrow D_Y = F^* D_Z$, the last expression is equal to $D_Y( \frac{\phi^1_\alpha(x)}{\eps} )$, which finishes the proof.
\qquad\end{proof}

The following lemma describes the transformation behaviour of the cell problems, which will later imply the main result:

\begin{lemma}
\label{lem:RMtransfocell}
Let $\phi:Y\longrightarrow Z$ be a coordinate transformation between the reference cells. As above, define $g_Y^{(x)}= \sum_{i,j} g_{ij}(x) \ud y^i \otimes \d y^j$ and $g_Z^{(x)}= \sum_{i,j} g_{ij}(x) \ud z^i \otimes \d z^j$ and assume that there exists a $\lambda>0$ such that both metrics are related by $\lambda g_Z^{(x)}= \phi_* g_Y^{(x)}$. Furthermore, assume that $D_Z=\phi_* D_Y$. 

For given vector fields $Q\in L^2TY$ and $H\in L^2TZ$ consider the generalised cell problems: Find $w_Y^Q \in H^1_\#(Y)/ \R$ and $w_Z^H \in H^1_\#(Z)/ \R$ such that for fixed $x\in M$
\begin{subequations}
\begin{gather}
-\div_Y^{(x)}( D_Y\grad_Y^{(x)} w_Y^Q  ) = \div_Y^{(x)}(D_Y Q) \quad \text{in } Y \label{eq:RMgencellY}\\
y \longmapsto w_Y^Q(x,y) \quad \text{is $Y$-periodic}
\end{gather}
\end{subequations}
and
\begin{subequations}
\label{eq:RMgencellZ}
\begin{gather}
-\div_Z^{(x)}( D_Z\grad_Z^{(x)} w_Z^H  ) = \div_Z^{(x)}(D_Z H) \quad \text{in } Z \\
z \longmapsto w_Z^H(x,z) \quad \text{is $Z$-periodic}.
\end{gather}
\end{subequations}
Then it holds
\[
\lambda \phi_* w_Y^Q = w_Z^{\phi_* Q} \quad \text{as well as} \quad  \phi_*( \grad_Y^{(x)} w_Y^Q ) =  \grad_Z^{(x)} w_Z^{\phi_* Q}.
\]
\end{lemma}
\begin{proof}
Keep the relations $\div_Z^{(x)} \circ \phi_* = \phi_* \circ \div_Y^{(x)} $ as well as $\grad_Z^{(x)} \circ \phi_* = \frac{1}{\lambda}\phi_*\circ \grad_Y^{(x)}$ in mind, which hold due to the asserted relation between the metrics on $Y$ and $Z$, see e.g.\ \cite{ana3}.
Application of $\phi_*$ to equation \eqref{eq:RMgencellY} yields
\begin{align*}
-\phi_*[\div_Y^{(x)}( D_Y\grad_Y^{(x)} w_Y^Q  )] &= \phi_*{\div_Y^{(x)}(D_Y Q)} \\
\Longleftrightarrow \quad - \div_Z^{(x)} ( D_Z  \grad_Z^{(x)} (\lambda \phi_* w_Y^Q) ) &= \div_Z^{(x)}(D_Z \phi_* Q)  
\end{align*}
Since the solution of \eqref{eq:RMgencellZ} (with $H=\phi_* Q$) is unique, we obtain $\lambda \phi_* w_Y^Q = w_Z^{\phi_* Q}$. Application of $\grad_Z^{(x)}$ to both sides of this identity finally gives
\[
\grad_Z^{(x)} w_Z^{\phi_* Q} = \lambda\grad_Z^{(x)} ( \phi_* w_Y^Q  ) = \phi_*( \grad_Y^{(x)} w^Q_Y ).
\]
\qquad\end{proof}

\label{rem:RMcellpro}
An analogous argument as in the proof above shows that for $\alpha, \beta \in \R$ and $H_1,H_2 \in L^2TZ$ it holds
\[
\alpha w_Z^{H_1} + \beta w_Z^{H_2} = w_Z^{\alpha H_1 + \beta H_2}.
\]




In the sequel, we will use the index notation of coordinate transformations in differential geometry (see e.g. \cite{zeidler4}): Let $\phi^1_\alpha \in \mathscr{A}_1, \phi^2_\beta \in \mathscr{A}_2$ such that $U^1_\alpha \cap U^2_\beta \not= \emptyset$. Writing $\phi^1_\alpha=(x^1, \dots, x^n)$ as well as $\phi^2_\beta=(\tilde{x}^1,\dots, \tilde{x}^n)$ for the components of the charts, one uses the notation $\frac{\p \tilde{x}^i}{\p x^j}(x)$ to denote the $ij$-th entry of the Jacobian matrix of $\phi^2_\beta\circ (\phi^1_\alpha)^{-1}$ at $\phi^1_\alpha(x)$, $x\in U^1_\alpha\subset M$. 

Note two peculiarities due to our Assumptions \ref{ass:RMequiv}: First, $\frac{\p \tilde{x}^i}{\p x^j}(x)$ corresponds to the $ij$-th entry in the matrix representation of the linear map $F$; and second, due to Lemma \ref{lem:RMFunique}, this value is constant on all of $M$. Furthermore, we will switch between the interpretations of $\phi^2_\beta \circ (\phi^1_\alpha)^{-1}$ being a coordinate transformation for $M$ and for $Y$ without using a specific notation.

\begin{lemma}
\label{lem:RMtransfometric}
It holds
\[
(F|_Y)_* g_Y^{(x)} = g_Z^{(x)}.
\]
\end{lemma}
\begin{proof}
By the usual transformation rules for tensor fields, the Riemannian metric $g$ on $M$ has the two local representations
\begin{align*}
g&=\sum_{i,j} g_{ij} \ \d x^i\otimes \d x^j \quad \text{as well as} \quad
g= \sum_{i,j} \tilde{g}_{ij} \ \d \tilde{x}^i \otimes \tilde{x}^j,
\end{align*}
where the coefficients are related via the identity $\tilde{g}_{lk} = \sum_{i,j} g_{ij} \frac{\p x^i}{\p \tilde{x}^l} \frac{\p x^j}{\p \tilde{x}^k}$. By the construction of the induced metric on the reference cell, we obtain the metrics
\begin{align*}
g_Y^{(x)}& = \sum_{i,j} g_{ij}(x) \ud y^i \otimes \d y^j    \quad\text{and} \quad
g_Z^{(x)}= \sum_{i,j} \tilde{g}_{ij}(x) \ud z^i \otimes \d z^j.
 \end{align*}
Let $X\in TZ$ be a vector field in $Z$, with local representation $X=\sum_i X^i \frac{\p}{\p z^i} $. By the transformation rules for vector fields, the local representation of $F^* X\in TY$ is given by $\sum_i ( \sum_l  X^l \frac{\p x^i}{\p \tilde{x}^l}  ) \frac{\p}{\p y^i}$. We obtain for this $X$ and a similar vector field $Y\in TZ$
\begin{align*}
[F_* g_Y^{(x)}](X,Y) &= g_Y^{(x)}(F^*X, F^*Y ) = \sum_{i,j} \Bigl(  g_{ij}(x)\cdot[\sum_l X^l \frac{\p x^i}{\p \tilde{x}^l}] \cdot [\sum_k Y^k \frac{\p x^j}{\p \tilde{x}^k}] \Bigr) \\  
& = \sum_{l,k} \tilde{g}_{lk}(x)  X^l Y^k = g_Z^{(x)}(X,Y).
\end{align*}
This shows that $F_* g_Y^{(x)}= g_Z^{(x)}$.
\qquad\end{proof}


\textit{Proof of the Theorem:}
Taking a look at the proof of Theorem \ref{thm:RMhomogeq} (and especially step 2), we have to show that the expression $\frac{1}{|Y|}\int_Y D[ \grad_M u  +\sum_{i=1}^n  (\grad_M u)^i(\grad_Y^{(x)}  w_i )_M ]   \ud y$ has an "appropriate" transformation behaviour. 
 Put in the framework used in this section, we have to compare the terms (see the expression $K$ in the proof mentioned above)
\begin{gather*}
\sum_{i,k} D_Y (\delta^i_k + (\grad_Y^{(x)} w_Y^{e_i})^k) (\grad_M u)^i  \frac{\p}{\p x^k} \quad\text{and}\quad
\sum_{i,k} D_Z (\delta^i_k + (\grad_Z^{(x)} w_Z^{e_i})^k) (\grad_M u)^i \frac{\p}{\p \tilde{x}^k}.
\end{gather*}
Note that here $(\grad_M u)^i$ in both formulas does \emph{not} signify the same mathematical expression! In the first formula, $(\grad_M u)^i$ denotes the $i$-th component with respect to the local basis $\frac{\p}{\p x^k}$, whereas in the second formula the same term is the $i$-th component with respect to the local basis $\frac{\p}{\p \tilde{x}^k}$. To avoid this notational confusion, we use the representation $\grad_M u = \sum_i X^i \frac{\p}{\p x^i}$. Then, by the transformation rules for vector fields, $\grad_M u = \sum_i ( \sum_l  X^l \frac{\p \tilde{x}^i}{\p x^l} ) \frac{\p}{\p \tilde{x}^i}$. Now we see that $(\grad_M u)^i = X^i$ in the first expression, and $(\grad_M u)^i= \sum_l X^l \frac{\p \tilde{x}^i}{\p x^l}=: \tilde{X}^i$ in the second.

\textbf{Step 1} Transformation of the individual terms:\\
We have 
\begin{align*}
F_*[\sum_{l,m} \delta^m_l (\grad_M u)^l \frac{\p}{\p x^m}   ] & = F_*[\sum_{l,m} \delta^m_l X^l \frac{\p}{\p x^m}   ] = \sum_{l,m,i,k} \delta^k_i X^l \frac{\p \tilde{x}^i}{\p x^l} \frac{\p x^m}{\p \tilde{x}^k} \frac{\p}{\p x^m} 
 = \sum_{i,k} \delta^k_i \tilde{X}^i \frac{\p}{\p \tilde{x}^k}.
\end{align*}
Due to Lemma \ref{lem:RMtransfometric}, we can use the transformation Lemma \ref{lem:RMtransfocell} for the cell problems (with $\lambda=1$) and Remark \ref{rem:RMcellpro} to obtain
\begin{align*}
F_*[ \sum_{i,k} (\grad_Y^{(x)} w_Y^{e_i} )^k X^i \frac{\p}{\p x^k} ] &= \sum_{i,k,m}  (\grad_Y^{(x)} w_Y^{e_i} )^k   X^i  (\frac{\p \tilde{x}^m}{\p {x}^k}  \frac{\p}{\p \tilde{x}^m}  )\\
&= \sum_{i,m} \underbrace{ \sum_k \frac{\p \tilde{x}^m}{\p x^k} (\grad_Y^{(x)} w_Y^{e_i} )^k}_{=(F_* \grad_Y^{(x)} w_Y^{e_i} )^m} X^i \frac{\p}{\p \tilde{x}^m} \\
& = \sum_{i,m} (\grad_Z^{(x)}  w_Z^{F_* e_i} )^m X^i\frac{\p}{\p \tilde{x}^m} \\
& = \sum_{i,m} (\grad_Z^{(x)}  w_Z^{\sum_k  \frac{\p \tilde{x}^k}{\p x^i} e_k} )^m X^i \frac{\p}{\p \tilde{x}^m} \\
& = \sum_{k,m} (\grad_Z^{(x)} w_Z^{e_k})^m (\sum_i  X^i \frac{\p \tilde{x}^k}{\p x^i} )\frac{\p}{\p \tilde{x}^m} \\
& = \sum_{k,m} (\grad_Z^{(x)} w_Z^{e_k})^m \tilde{X}^k \frac{\p}{\p \tilde{x}^m}.
\end{align*}

\textbf{Step 2} Transformation of the integrals:\\
Keeping in mind $F_* D_Y = D_Z$, $F_*(\d y^1 \dots \d y^n) = |\det(F^{-1})| \d z^1 \dots \d z^n$ and the formulas derived in step 1, we can apply the pushforward $F_*$ to the integral $\frac{1}{|Y|} \int_Y D_Y [ \grad_M u  +\sum_{i=1}^n  (\grad_M u)^i(\grad_Y^{(x)}  w_Y^{e_i} )_M ]   \ud y$ to obtain
\begin{align*}
F_* &  \Bigl[  \frac{1}{|Y|} \int_Y \Bigl(\sum_{i,k} D_Y (\delta^i_k + (\grad_Y^{(x)} w_Y^{e_i} )^k)(\grad_M u)^i \frac{\p}{\p x^k} \Bigr) \ud y^1 \dots \d y^n \Bigr] \\
& = \frac{1}{|Y|} \int_Z F_* D_Y \Bigl(  F_*[\sum_{i,k} \delta^i_k X^i \frac{\p}{\p x^k}  ]  + F_*[  \sum_{i,k}   \sum_{i,k} (\grad_Y^{(x)} w_Y^{e_i} )^k X^i \frac{\p}{\p x^k}   ]      \Bigr) \, F_*[\d y^1 \dots \d y^n  ] \\
& = \frac{1}{|Y|}  |\det(F^{-1})| \int_Z D_Z \Bigl( \sum_{l,m} \delta^l_m \tilde{X}^m \frac{\p}{\p \tilde{x}^l} + (\grad_Z^{(x)} w_Z^{e_m})^l \tilde{X}^m \frac{\p}{\p \tilde{x}^l}   \Bigr) \ud z^1 \dots \d z^n.
\end{align*}
Since $\frac{|\det(F^{-1})|}{|Y|}= \frac{1}{|Z|}$ by the transformation rule for integrals, we finally get that the last term is equal to
\[
\frac{1}{|Z|} \int_Z D_Z ( \sum_{i,k} \delta^i_k + (\grad_Z^{(x)} w_Z^{e_i}  )^k )\tilde{X}^i \frac{\p}{\p \tilde{x}^k} \ud z,
\]
which is (written with respect to the local basis $\frac{\p}{\p \tilde{x}^k}$) nothing else but \[\frac{1}{|Z|} \int_Z  D_Z[  \grad_M u  + \sum_{i=1}^n (\grad_M u)^i (\grad_Z^{(x)}  w_Z^{e_i} )_M  ]   \ud z. \] Thus we see that the expression constituting the homogenised problem is invariant under a change of the atlas which satisfies Assumptions~\ref{ass:RMequiv}.


\bibliography{../../literatur}







\end{document}